\documentclass[12pt]{article}

\usepackage{cite}

\usepackage{times}

\usepackage{amsmath,amssymb,amsfonts}
\usepackage{mathrsfs}
\usepackage{graphicx}
\usepackage{epstopdf}
\usepackage{array}
\usepackage{float}
\usepackage{tikz}
\usepackage{pgflibraryarrows}
\usepackage{pgflibrarysnakes}
\usepackage{lineno}

\usetikzlibrary{decorations.pathreplacing}
\allowdisplaybreaks[4]

\newtheorem{theorem}{Theorem}
\newtheorem{lemma}{Lemma}

\newenvironment{proof}[1][Proof]{\noindent\text{#1}\,.\,\,}{\hfill  $\square$}

\newcounter{lastnote}

\title{On the Metric Dimension of Generalized Petersen Graphs $P(n,3)$}

\author{Rui Gao$^{1}$, Yingqing Xiao$^{2}$, and Zhanqi Zhang$^{3\ast}$
	\\
	\small{$^{1}$School of Statistics and Mathematics,}\\
	\small{Shandong University of Finance and Economics,}\\
	\small{Jinan, 250014, China}\\
	\small{$^{2}$School of Mathematics, Hunan University,}\\
	\small{Changsha, 410082, China}\\
	\small{$^{3}$School of Mathematics and Statistics,}\\
	\small{Hunan University of Science and Technology,}\\
	\small{Xiangtan, 411201, China}\\
	\small{$^\ast$Corresponding author: Zhanqi Zhang(e-mail: rateriver@sina.com).}
}

\date{}

\begin{document}
\baselineskip24pt
\maketitle

\begin{abstract}
The metric dimension of a graph $G$ is defined as the minimum number of vertices in a subset $S\subset V(G)$ such that all other vertices are uniquely determined by their distances to the vertices in $S$, and is denoted by $\dim(G)$. 
In this paper, we study the metric dimension of generalized Petersen graphs $P(n,3)$. The notions of good and bad vertices, which are introduced in
Imran et al. (2014, Ars. Combinatoria 117, 113-130), are instrumental in determining the lower bound of the metric dimension for certain types of graphs.
We propose an approach, based on these notions, to determine the lower bound of $\dim(P(n,3))$. Moreover, we shall prove that $\dim(P(n,3))=4$,
where $n\equiv2,3,4,5 \,\,(\text{mod}\,\, 6)$ and is sufficiently large.
\end{abstract}



\section{Introduction}
Let $G=(V(G),E(G))$ be a simple undirected connected graph. The distance between two vertices $u,v$ is the length of a shortest path 
between them, and is denoted by $d(u,v)$. 
A vertex $w$ is said to \textit{resolve} or \textit{distinguish} $u$ and $v$ if $d(w,u)\neq d(w,v)$.
For an ordered set $W=\{w_1,\cdots,w_k\}$ of $k$ distinct vertices and a vertex $z$ in $G$,
we refer to the $k$-tuple $r(z|W)=(d(z,w_1),d(z,w_2),\cdots,d(z,w_k))$ as the \textit{metric representation} of $z$ with respect to $W$. 
The set $W$ is said to \textit{resolve} or \textit{distinguish} a pair of vertices $u,v$ if $r(u|W)\neq r(v|W)$.
Furthermore, $W$ is called a \textit{resolving set} of $G$ if $r(u|W)=r(v|W)$ implies that $u=v$ for all $u,v\in V(G)$.
A resolving set containing a minimum number of vertices is called a \textit{metric basis} of $G$, and
its cardinality the \textit{metric dimension} of $G$, denoted by $\mathrm{dim}(G)$.

Inspired by the problem of pinpointing the exact location of an intruder within a network, Slater introduced the notion of metric dimension in \cite{Slater}. Harary and Melter independently proposed the concept of metric dimension in \cite{Harary}.
It was proved that the metric dimension is a NP-hard graph invariant \cite{Khuller}. Mathematicians have undertaken extensive studies on the metric dimension of numerous graphs exhibiting unique structural properties, for example, the wheel \cite{Buczkowski}, the fan \cite{Hernando}, the Jahangir graph \cite{Tomescu}, the unicyclic graph \cite{Sedlar}, and the circulant graph \cite{Vetrik1,Gao,Vetrik2,Tapendra}.
Random graph models, which define the probability distributions over graph structures and often involve some generative mechanism, 
are more suitable for simulating real-world networks than deterministic graphs. Recently, the metric dimension of Erd\H{o}s-R\'{e}nyi random graphs \cite{Bollobas} and random trees and forests \cite{Mitsche} have been characterized.
For more detailed information on the history, applications, and future research directions of metric dimension, we refer to \cite{Tillquist}.

The concept of generalized Petersen graphs $P(n,m)$ was initially introduced in \cite{Watkins}, here we require that $n\geq3$ and $1\leq m\leq\lfloor\frac{n-1}{2}\rfloor$.
We introduce some results concerning the metric dimension of generalized Petersen graphs. In \cite{Javaid1} it was proved that $\dim(P(n,2))=3$ for $n\geq5$. In \cite{Imran1}, Imran et al. studied the metric dimension of generalized Petersen graphs $P(n,3)$, and obtained that for sufficiently large $n$,

\begin{align*}
	&\dim(P(n,3))
	\begin{cases}
		=4,\quad   & \text{if}\,\, n\equiv0 \,\,(\text{mod}\,\, 6), \\
		=3,\quad   & \text{if}\,\, n\equiv1 \,\,(\text{mod}\,\, 6),\\
		\leq5,\quad   & \text{if}\,\, n\equiv2 \,\,(\text{mod}\,\, 6), \\
		\leq4,\quad   & \text{if}\,\, n\equiv3,4,5 \,\,(\text{mod}\,\, 6). \\
	\end{cases}
\end{align*}
In \cite{Naz}, Naz et al. studied the metric dimension of generalized Petersen graphs $P(n,4)$, and obtained that for sufficiently large $n$,

\begin{align*}
	&\dim(P(n,4))
	\begin{cases}
		=3,\quad   & \text{if}\,\, n\equiv0 \,\,(\text{mod}\,\, 4), \\
		\leq4,\quad   & \text{if}\,\, n\equiv1,2 \,\,(\text{mod}\,\, 4),\\
		\leq4,\quad   & \text{if}\,\, n=4k+3\,\,\text{and}\,\,k\,\,\text{is odd}, \\
		=4,\quad   & \text{if}\,\, n=4k+3\,\,\text{and}\,\,k\,\,\text{is even}. \\
	\end{cases}
\end{align*}
In \cite{Shao}, Shao et al. calculated the values of $\dim(P(n,3))$ and $\dim(P(n,4))$ when $n$ is relatively small, and also explored the metric dimensions of $P(2n,n)$ and $P(3n,n)$.
In \cite{Javaid2}, Javaid et al. considered the generalized Petersen graphs $P(2n+1,n)$, and obtained that $\dim(P(2n+1,n))=3$ for $n\geq2$.
It was proved in \cite{Ahmad2} that the generalized Petersen graphs $P(2n,n-1)$ have metric dimension equal to $3$ for odd $n\geq3$, and equal to $4$ for even $n\geq4$. In \cite{Imran2}, Imran et al. revisited the generalized Petersen graphs $P(2n,n)$ and deduced that 
$P(2n,n)$ have metric dimension $3$ when $n$ is even and $4$ otherwise.

This paper is devoted to the study of the metric dimension of generalized Petersen graphs $P(n,3)$. In this paper, we always assume that $n$ is sufficiently large, for example, $n\geq36$.
$P(n,3)$ is an important family of cubic graphs having vertex-set

$$
V=\{u_1,u_2,\ldots,u_n,v_1,v_2,\ldots,v_n\},
$$
and edge-set 

$$
E=\{u_iu_{i+1},u_iv_i,v_iv_{i+3}:1\leq i\leq n\}.
$$
Index $i$ is called a \textit{subscript} of $u_i$ and $v_i$, and is taken modulo $n$.
The subgraph that comprises vertex-set $\{u_1,u_2,\ldots,u_n\}$ and 
edge-set $\{u_iu_{i+1}:1\leq i\leq n\}$ is referred to as the \textit{outer cycle}. If $n\equiv 0\,(\text{mod}\,3)$, then $\{v_1,v_2,\ldots,v_n\}$ induces $3$ cycles of length $\frac{n}{3}$, otherwise it induces a cycle of length $n$ with $v_iv_{i+3},1\leq i\leq n$ as edges. We call the cycles induced by $\{v_1,v_2,\ldots,v_n\}$ the \textit{inner cycles}. $u_i$ is called the \textit{corresponding vertex} of $v_i$ on the outer cycle, and vice versa. 

\section{Lower bounds for the metric dimension}
For the sake of brevity we define 

$$
\forall\, L\in \mathbb{N}:\quad f(L):=L-2\left\lfloor\frac{L}{3}\right\rfloor.
$$
Although $f$ is not a monotonic function, we check that $f(L_1)\geq f(L_2)$ whenever $L_1\geq L_2+2$. There is a convenient method to compute the value of $f$; that is, suppose that $L=3m+i$, where $i=0,1,2$, then $f(L)=m+i$.

The \textit{clockwise distance} from $u_i$ to $u_j$, denoted by $d^{*}(u_i,u_j)$, is defined as the number of edges that must be crossed in the outer cycle to move from $u_i$ to $u_j$ in a clockwise direction. For example, $d^{*}(u_1,u_n)=n-1$ and $d^{*}(u_n,u_1)=1$. This definition can be extended to any two vertices with different subscripts $i\neq j$, i.e., 
$d^{*}(u_i,v_j)=d^{*}(v_i,u_j)=d^{*}(v_i,v_j)=d^{*}(u_i,u_j)$.
Let $A$ be the set of all vertices on the outer cycle whose clockwise distance from $u_1$ is congruent to $1$ modulo $3$, $B$ the set of all vertices on the outer cycle whose clockwise distance from $u_1$ is congruent to $2$ modulo $3$, and let $C$ be the set of all vertices on the outer cycle, except for $u_1$, whose clockwise distance from $u_1$ is congruent to $0$ modulo $3$. For example, when $n=6k+3$, we have that $A=\{u_2,u_5,u_8,\ldots,u_{6k-1},u_{6k+2}\}$, $B=\{u_3,u_6,u_9,\ldots,u_{6k},u_{6k+3}\}$ and $C=\{u_4,u_7,u_{10},\ldots,u_{6k-2},u_{6k+1}\}$.

Now we introduce the notion of good and bad vertices, which is used in \cite{Imran1} to obtain the lower bounds for $\dim(P(n,3))$.
Let $w$ be a vertex of $P(n,3)$ and $W$ a subset of $V(P(n,3))$.
A vertex $u_i$ on the outer cycle is called a \textit{good vertex} for $w$ if $u_i$ and $u_{i+2}$ have equal distance to $w$; otherwise $u_i$ is called a \textit{bad vertex} for $w$.
If $u_i$ is a bad vertex for $w$, we also say that $w$ can \textit{recognize} $u_i$, and write $w\xrightarrow{\text{Reg}}u_i$ to denote this situation.
$W$ is said to \textit{recognize} $u_i$ if at least a member of $W$ can recognize $u_i$. It is worth noticing that, each resolving set of $P(n,3)$ can recognize all vertices on the outer cycle.
Let $A(u_i)$ denote the subset of $A$ consisting of all the vertices capable of recognizing $u_i$. $B(u_i)$ and $C(u_i)$ can be similarly defined.
The following lemma provides a means of finding vertices, for which a given vertex is good.

\begin{lemma}\label{xxx}
Let $1\leq i\leq n-1$. If $u_{j-i}$ is a good vertex for $u_1$, then $u_j$ is a good vertex for $u_{1+i}$.
\end{lemma}
\begin{proof}
Due to the rotational symmetry of $P(n,3)$ we have 

$$
d(u_{j-i},u_1)=d(u_j,u_{1+i})\quad\text{and}\quad d(u_{j-i+2},u_1)=d(u_{j+2},u_{1+i}),
$$
and, since $u_{j-i}$ and $u_{j-i+2}$ have equal distance to $u_1$, it follows that $u_j$ and $u_{j+2}$ have equal distance to $u_{1+i}$.
\end{proof}

\subsection{Case when $n=6k+3$}\label{subsection3}
In this subsection, we always assume $n=6k+3$ and $k\geq6$. For $n\equiv 3\,(\text{mod}\,6)$, the distance between two vertices in $P(n,3)$ is $d(u_i,v_j)=f(L)+1$, and 
\begin{align*}
	&d(u_i,u_j)=
	\begin{cases}
		L,\quad   & \text{if}\,\, L\leq2, \\
		f(L)+2,\quad   & \text{if}\,\, L\geq3,
	\end{cases}
	 \\
	&d(v_i,v_j)=
	\begin{cases}
		f(L),\quad   & \text{if}\,\, L\equiv 0\,(\text{mod}\,3), \\
		f(L)+2,\quad   & \text{if}\,\, L\equiv 1,2\,(\text{mod}\,3),
	\end{cases}
\end{align*}
where $L=|i-j|\wedge (n-|i-j|)$.
According to the distance formula, we can find out all good vertices for $u_1$:
$$
u_5,u_6,\ldots,u_{3i-1},u_{3i},\ldots,u_{3k-1},u_{3k},u_{3k+3},u_{3k+4},\ldots,u_{3k+3j},u_{3k+3j+1},\ldots,u_{6k-3},u_{6k-2},u_{6k+3}.
$$
Alternatively, $u_1$ can recognize all vertices on the outer cycle, except for the above. 
Noticing that the set of good vertices for $v_1$ is deduced from that for $u_1$ by adding $4$ new vertices $u_2,u_3,u_{6k},u_{6k+1}$. 
It follows from the rotational symmetry that $u_i$ can recognize more vertices on the outer cycle than $v_i$ for each $i$.
There is another point worth noticing: suppose that a set of vertices $W$ cannot recognize a vertex, say $u_j$, on the outer cycle. If we replace one or several members of $W$, which are on the outer cycle, by the corresponding vertices on the inner cycle, then the newly obtained set is unable to recognize $u_j$ either. 

Now we introduce a method for identifying all the members in $A$ that can recognize a given vertex. 
Choose a vertex on the outer cycle, say $u_{3i-1}$, $2\leq i\leq k$. Subtract $1,4,7,\ldots,6k-2,6k+1$ from the subscript of this vertex in sequence and we obtain
$$
u_{(3i-1)-1},u_{(3i-1)-4},u_{(3i-1)-7},\ldots,u_{(3i-1)-(6k+1)}.
$$
Among them, one can verify that $u_{(3i-1)-(3i+4)},u_{(3i-1)-(3i+7)},\ldots,u_{(3i-1)-(3k+3i-2)}$ are good for $u_1$. 
It follows from Lemma \ref{xxx} that $u_{3i-1}$ is good for $u_{1+(3i+4)},u_{1+(3i+7)},\ldots,u_{1+(3k+3i-2)}$.
Removing these vertices from set $A$, the newly obtained set consists of all the vertices that can recognize $u_{3i-1}$.
Note that the methods for finding sets $B(u_i)$ or $C(u_i)$
are similar to that of sets $A(u_i)$, with the first step being the only difference. To find $B(u_i)$, we need to subtract $2,5,8,\ldots,6k-1,6k+2$ from the subscript of $u_i$; to find $C(u_i)$, we need to subtract $3,6,9,\ldots,6k-3,6k$ from the subscript of $u_i$.
Using this method, we obtain Table \ref{table1}.
\begin{table}[h]
	\centering
	\begin{tabular}{|c|c|c|}
		\hline
		\textbf{Set} & \textbf{Vertices} & \textbf{Range}\\\hline
		\hline
		$A(u_{3i-1})$ & $u_2,u_5,\ldots,u_{3i+2},u_{3k+3i+2},\ldots,u_{6k+2}$&$2\leq i\leq k$ \\
		\hline
		$A(u_{3i})$ & $u_{3i-1},u_{3i+2},\ldots,u_{3k+3i+2}$&$2\leq i\leq k$ \\
		\hline
		$A(u_{3k+3i})$ & $u_2,\ldots,u_{3i-1},u_{3k+3i-1},u_{3k+3i+2},\ldots,u_{6k+2}$ &$1\leq i\leq k-1$\\
		\hline
		$A(u_{3k+3i+1})$ & $u_{3k+3i-1},u_{3k+3i+5}$&$1\leq i\leq k-1$ \\
		\hline
		$B(u_{3i-1})$ & $u_{3i-3},u_{3i+3}$&$2\leq i\leq k$ \\
		\hline
		$B(u_{3i})$ & $u_3,u_6,\ldots,u_{3i+3},u_{3k+3i+3},\ldots,u_{6k+3}$&$2\leq i\leq k$ \\
		\hline
		$B(u_{3k+3i})$ & $u_{3i},u_{3i+3},\ldots,u_{3k+3i+3}$&$1\leq i\leq k-1$ \\
		\hline
		$B(u_{3k+3i+1})$ & $u_3,\ldots,u_{3i},u_{3k+3i},u_{3k+3i+3},\ldots,u_{6k+3}$&$1\leq i\leq k-1$ \\
		\hline
		$C(u_{3i-1})$ & $u_{3i-2},u_{3i+1},\ldots,u_{3k+3i+1}$&$2\leq i\leq k$ \\
		\hline
		$C(u_{3i})$ & $u_{3i-2},u_{3i+4}$&$2\leq i\leq k$ \\
		\hline
		$C(u_{3k+3i})$ & $u_{3k+3i-2},u_{3k+3i+4}$&$1\leq i\leq k-1$ \\
		\hline
		$C(u_{3k+3i+1})$ & $u_{3i+1},u_{3i+4},\ldots,u_{3k+3i+4}$&$1\leq i\leq k-1$ \\
		\hline
	\end{tabular}
	\caption{The result when $n=6k+3$.}\label{table1}
\end{table}
Now we turn to the result:

\begin{theorem}\label{3}
If $n=6k+3$ and $k\geq6$, then $\dim(P(n,3))\geq4$.
\end{theorem}
\begin{proof}
Let us show that $\dim(P(n,3))$ has a lower bound of $4$. Suppose on the contrary that $W:=\{O,X,Y\}$ is a resolving set of $P(n,3)$.
We claim that no two vertices of $W$ have the same subscript; since if otherwise, those two vertices with the same subscript can only recognize $2k+6$ vertices on the outer cycle, and the other vertex must recognize the remaining $4k-3$ vertices on the outer cycle, which is impossible.
If $O$ is on the outer cycle, let $\widetilde{O}=O$; if otherwise, let $\widetilde{O}$ be the vertex corresponding to $O$ on the outer cycle. 
$\widetilde{X}$ and $\widetilde{Y}$ can be defined similarly. Then $\widetilde{W}:=\{\widetilde{O},\widetilde{X},\widetilde{Y}\}$ can still recognize all the vertices on the outer cycle. We
will discuss the problem in six cases, and within each case, explore what kind of positional relationship $\widetilde{O},\widetilde{X},\widetilde{Y}$ should have.
Apart from Cases $4,5$, we assume $\widetilde{O}=u_1$ by the rotational symmetry of $P(n,3)$. Note that for a vertex $u\in A$ and a vertex $u_i$, 
$u\xrightarrow{\text{Reg}}u_{i}$ if and only if $u\in A(u_i)$.

$\mathbf{Case\, 1.}$ $d^{*}(\widetilde{O},\widetilde{X})\equiv d^{*}(\widetilde{O},\widetilde{Y})\equiv0\,(\text{mod}\,3)$. It follows that $\{\widetilde{X},\widetilde{Y}\}$ can recognize $u_{3k+3i}$ for $1\leq i\leq5$. 
Suppose that $\widetilde{X}\xrightarrow{\text{Reg}}u_{3k+3}$. Since $C(u_{3k+3})\cap C(u_{3k+6})=\emptyset$, it follows that $\widetilde{X}$ cannot recognize $u_{3k+6}$, and therefore $\widetilde{Y}\xrightarrow{\text{Reg}}u_{3k+6}$.
Since $C(u_{3k+6})\cap C(u_{3k+9})=\emptyset$, it follows that $\widetilde{Y}$ cannot recognize $u_{3k+9}$, and therefore $\widetilde{X}\xrightarrow{\text{Reg}}u_{3k+9}$. Continuing in this manner, we see that $\widetilde{X}\xrightarrow{\text{Reg}}u_{3k+15}$, which implies that $\widetilde{X}\in C(u_{3k+3})\cap C(u_{3k+15})$, a contradiction.

$\mathbf{Case\, 2.}$ $d^{*}(\widetilde{O},\widetilde{X})\equiv d^{*}(\widetilde{O},\widetilde{Y})\equiv1\,(\text{mod}\,3)$. It follows that $\{\widetilde{X},\widetilde{Y}\}$ can recognize $u_{3k+3i+1}$ for $1\leq i\leq5$. 
Suppose that $\widetilde{X}\xrightarrow{\text{Reg}}u_{3k+4}$.
A proof completely analogous to that in Case 1 shows $\widetilde{X}\in A(u_{3k+4})\cap A(u_{3k+16})$, which is a contradiction.

$\mathbf{Case\, 3.}$ $d^{*}(\widetilde{O},\widetilde{X})\equiv d^{*}(\widetilde{O},\widetilde{Y})\equiv2\,(\text{mod}\,3)$.
It follows that $\{\widetilde{X},\widetilde{Y}\}$ can recognize $u_{3i-1}$ for $2\leq i\leq6$. 
Suppose that $\widetilde{X}\xrightarrow{\text{Reg}}u_{5}$.
A proof analogous to that in Case 1 shows $\widetilde{X}\in B(u_{5})\cap B(u_{17})$, which is a contradiction.

$\mathbf{Case\, 4.}$ $d^{*}(\widetilde{O},\widetilde{X}) \equiv0\,(\text{mod}\,3)$ and $d^{*}(\widetilde{O},\widetilde{Y})\equiv1\,(\text{mod}\,3)$. 
It is easy to see that $d^{*}(\widetilde{Y},\widetilde{O})\equiv d^{*}(\widetilde{Y},\widetilde{X})\equiv2\,(\text{mod}\,3)$, so that Case 4 can be reduced to Case 3. 

$\mathbf{Case\, 5.}$ $d^{*}(\widetilde{O},\widetilde{X}) \equiv0\,(\text{mod}\,3)$ and $d^{*}(\widetilde{O},\widetilde{Y})\equiv2\,(\text{mod}\,3)$. One can verify that  $d^{*}(\widetilde{Y},\widetilde{O})\equiv d^{*}(\widetilde{Y},\widetilde{X})\equiv1\,(\text{mod}\,3)$, so that Case 5 can be reduced to Case 2.

$\mathbf{Case\, 6.}$ $d^{*}(\widetilde{O},\widetilde{X}) \equiv1\,(\text{mod}\,3)$ and $d^{*}(\widetilde{O},\widetilde{Y})\equiv2\,(\text{mod}\,3)$. 
Suppose first that $\widetilde{X}\xrightarrow{\text{Reg}}u_{5}$. We declare that 
$\widetilde{X}\xrightarrow{\text{Reg}}u_{3k-1}$, since if otherwise 
$\widetilde{Y}\xrightarrow{\text{Reg}}u_{3k-1}$, then $\widetilde{Y}\in B(u_{3k-1})$. Since $B(u_{3k-1})\cap\big(B(u_{3k+7})\cup B(u_{3k+10})\big)=\emptyset$,
it follows that $\widetilde{X}$ must recognize $u_{3k+7}$ and $u_{3k+10}$, that is $\widetilde{X}\in A(u_{3k+7})\cap A(u_{3k+10})$, yielding a contradiction.  
The above implies that $\widetilde{X}\in A(u_5)\cap A(u_{3k-1})=\{u_2,u_5,u_8,u_{6k+2}\}$.

If $\widetilde{X}\in\{u_2,u_5,u_8\}$, the relation $\widetilde{X}\notin \cup_{i=1}^{k-1}A(u_{3k+3i+1})$
implies that $\widetilde{Y}\xrightarrow{\text{Reg}}u_{3k+3i+1}$ for each $1\leq i\leq k-1$, thus $\widetilde{Y}\in \cap_{i=1}^{k-1}B(u_{3k+3i+1})=\{u_3,u_{6k-3},u_{6k},u_{6k+3}\}$.
Consequently we have $\widetilde{Y}\notin B(u_{3k+6})$, indicating that $\widetilde{X}\xrightarrow{\text{Reg}}u_{3k+6}$, so $\widetilde{X}\neq u_8$. 
Now we observe that $\widetilde{X}\notin A(u_{3k})$, indicating that $\widetilde{Y}\xrightarrow{\text{Reg}}u_{3k}$,
and therefore $\widetilde{Y}\in\{u_3,u_{6k+3}\}$.
Note that $u_5\notin A(u_{3k+3})$, so if $\widetilde{X}=u_5$, then $\widetilde{Y}$ must be $u_3$.
At this point, we can summarize the following possibilities:
\begin{equation}\label{e1}
\widetilde{W}=\{u_1,u_2,u_3\},\,\{u_1,u_2,u_{6k+3}\}\,\,\text{or}\,\,\{u_1,u_5,u_3\}.
\end{equation}

If $\widetilde{X}=u_{6k+2}$, the relations $u_{6k+2}\notin \cup_{i=1}^{k-2}A(u_{3k+3i+1})$ and $u_{6k+2}\notin \cup_{i=2}^{k-1}A(u_{3i})$
imply that $\widetilde{Y}\xrightarrow{\text{Reg}}u_{3k+3i+1}$ for each $1\leq i\leq k-2$, and that $\widetilde{Y}\xrightarrow{\text{Reg}}u_{3i}$ for each $2\leq i\leq k-1$, so that $\widetilde{Y}\in\{u_3,u_{6k},u_{6k+3}\}$. Hence
\begin{equation}\label{e2}
\widetilde{W}=\{u_1,u_{6k+2},u_3\},\,\{u_1,u_{6k+2},u_{6k}\}\,\,\text{or}\,\,\{u_1,u_{6k+2},u_{6k+3}\}.
\end{equation}

Lastly, suppose that $\widetilde{Y}\xrightarrow{\text{Reg}}u_{5}$, that is $\widetilde{Y}\in \{u_3,u_9\}$. 
Clearly $\widetilde{Y}\notin B(u_{8})\cup B(u_{3k-1})$, so that both $u_{8}$ and $u_{3k-1}$ can be recognized by $\widetilde{X}$, and thus 
$\widetilde{X}\in A(u_{8})\cap A(u_{3k-1})=\{u_2,u_5,u_8,u_{11},u_{6k+2}\}$.
We point out that $\widetilde{Y}\neq u_9$; since if otherwise, then 
$\widetilde{X}\xrightarrow{\text{Reg}}u_{3k+3i+1}$ for $i=1,2$,
namely $\widetilde{X}\in A(u_{3k+4})\cap A(u_{3k+7})$, this is impossible.
Now we observe $\widetilde{X}\xrightarrow{\text{Reg}}u_{3k+6}$, which produces $\widetilde{X}\in\{u_2,u_5,u_{6k+2}\}$.
Based on this, we can deduce the following possibilities:  
\begin{equation}\label{e3}
\widetilde{W}=\{u_1,u_{6k+2},u_3\},\,\{u_{1},u_{2},u_3\}\,\,\text{or}\,\,\{u_{1},u_5,u_{3}\}.
\end{equation}

Now we only need to consider Possibilities (\ref{e1}), (\ref{e2}), and (\ref{e3}).
By the rotational symmetry, it suffices to discuss the cases where $\widetilde{W}=\{u_1,u_2,u_3\},\{u_1,u_5,u_3\}$.
In the first case, $W$ cannot resolve $u_{3k+2}$ and $u_{3k+5}$; in the second case, $W$ cannot resolve $u_{3k+1}$ and $u_{3k+8}$, which conflicts with our assumption. The proof is complete.
\end{proof}

\subsection{Case when $n=6k+4$}
In this subsection, we always assume $n=6k+4$ and $k\geq6$. For $n\equiv 4\,(\text{mod}\,6)$, the distance formulas for $u_i$ and $u_j$, as well as for $u_i$ and $v_j$, are the same as those when $n\equiv 3\,(\text{mod}\,6)$, but the distance formula for $v_i$ and $v_j$ is slightly different:
\begin{align*}
	&d(v_i,v_j)=
	\begin{cases}
		f(L),\quad   & \text{if}\,\, L\equiv 0\,(\text{mod}\,3), \\
		k+1,\quad   & \text{if}\,\, L=3k+1,\\
		f(L)+2,\quad   & \text{elsewise}, \\
	\end{cases}
\end{align*}
where $L=|i-j|\wedge (n-|i-j|)$.
According to the distance formula, we can find out all good vertices for $u_1$:
\begin{equation*}
\resizebox{.98\hsize}{!}
	{$\displaystyle u_5,u_6,\ldots,u_{3i-1},u_{3i},\ldots,u_{3k-1},u_{3k},u_{3k+2},u_{3k+4},u_{3k+5},\ldots,u_{3k+3j+1},u_{3k+3j+2},\ldots,u_{6k-2},u_{6k-1},u_{6k+4}.$}
\end{equation*}
It is worth noticing that the set of good vertices for $v_1$ is deduced from that for $u_1$ by adding $4$ new vertices $u_2,u_3,u_{6k+1},u_{6k+2}$. Hence each vertex on the outer cycle can recognize more vertices than the corresponding one on the inner cycle. 
Utilizing the approach delineated in Section \ref{subsection3}, for any given vertex, it is possible to identify all the vertices within sets $A$, $B$, or $C$ that are capable of recognizing it (see Table \ref{table2} for details).
\begin{table}[h]
	\centering
	\begin{tabular}{|c|c|c|}
		\hline
		\textbf{Set} & \textbf{Vertices} & \textbf{Range}\\\hline
		\hline
		$A(u_{3i-1})$ & $u_2,u_5,\ldots,u_{3i+2}$&$2\leq i\leq k+1$ \\
		\hline
		$A(u_{3i})$ & $u_{3i-1},u_{3i+2},\ldots,u_{3k+3i+2}$&$2\leq i\leq k$ \\
		\hline
		$A(u_{3k+3i+1})$ & $u_2,\ldots,u_{3i-1},u_{3k+3i-1},u_{3k+3i+5}$ &$1\leq i\leq k-1$\\
		\hline
		$A(u_{3k+3i+2})$ & $u_{3i+2},u_{3i+5},\ldots,u_{3k+3i+5}$&$1\leq i\leq k-1 $\\
		\hline
		$B(u_{3i-1})$ & $u_{3i-3},u_{3i+3},u_{3k+3i+3},\ldots,u_{6k+3}$&$2\leq i\leq k$ \\
		\hline
		$B(u_{3k+2})$ & $u_{3k},u_{3k+6}$& \\
		\hline
		$B(u_{3i})$ & $u_3,u_6,\ldots,u_{3i+3}$&$2\leq i\leq k$ \\
		\hline
		$B(u_{3k+3i+1})$ & $u_{3k+3i},u_{3k+3i+3},\ldots,u_{6k+3}$&$1\leq i\leq k-1$ \\
		\hline
		$B(u_{3k+3i+2})$ & $u_3,\ldots,u_{3i},u_{3k+3i},u_{3k+3i+6}$&$1\leq i\leq k-1 $\\
		\hline
		$C(u_{3i-1})$ & $u_{3i-2},u_{3i+1},\ldots,u_{3k+3i+1}$&$2\leq i\leq k+1$ \\
		\hline
		$C(u_{3i})$ &$ u_{3i-2},u_{3i+4},u_{3k+3i+4},\ldots,u_{6k+4}$&$2\leq i\leq k$ \\
		\hline
		$C(u_{3k+3i+1})$ & $u_{3i+1},u_{3i+4},\ldots,u_{3k+3i+4}$& $1\leq i\leq k-1$\\
		\hline
		$C(u_{3k+3i+2})$ & $u_{3k+3i+1},u_{3k+3i+4},\ldots,u_{6k+4}$&$1\leq i\leq k-1$ \\
		\hline
	\end{tabular}
	\caption{The result when $n=6k+4$.}\label{table2}
\end{table}

\begin{theorem}\label{4}
	If $n=6k+4$ and $k\geq6$, then $\dim(P(n,3))\geq4$.
\end{theorem}
\begin{proof}
Suppose on the contrary that $W:=\{O,X,Y\}$ is a resolving set of $P(n,3)$.
We claim that no two vertices of $W$ have the same subscript; since if otherwise, those two vertices with the same subscript can only recognize $2k+6$ vertices on the outer cycle, and the other vertex must recognize the remaining $4k-2$ vertices on the outer cycle, which is impossible.
Let $\widetilde{O},\widetilde{X},\widetilde{Y}$ be defined in the same way as $\widetilde{O},\widetilde{X},\widetilde{Y}$ were in Theorem \ref{3}. We see that $\widetilde{W}:=\{\widetilde{O},\widetilde{X},\widetilde{Y}\}$ can still recognize all the vertices on the outer cycle.
Let us discuss this problem in several cases. Apart from Cases $2,5,6$, we assume $\widetilde{O}=u_1$ by the rotational symmetry of $P(n,3)$.

$\mathbf{Case\, 1.}$ $d^{*}(\widetilde{O},\widetilde{X})\equiv d^{*}(\widetilde{O},\widetilde{Y})\equiv0\,(\text{mod}\,3)$. By symmetry, assume that $\widetilde{X}\xrightarrow{\text{Reg}}u_{5}$. 
It follows from the relation $C(u_5)\cap C(u_{6k-1})=\emptyset$ that $\widetilde{Y}\xrightarrow{\text{Reg}}u_{6k-1}$. Notice that $u_{3k}$ must be recognized by $\widetilde{X}$ or $\widetilde{Y}$. 
If $\widetilde{X}\xrightarrow{\text{Reg}}u_{3k}$, then 
$\widetilde{X}\in C(u_{5})\cap C(u_{3k})=\{u_{3k-2},u_{3k+4}\}$, so that  $\widetilde{Y}\xrightarrow{\text{Reg}}u_{3k-3}$, and therefore 
$\widetilde{Y}\in C(u_{3k-3})\cap C(u_{6k-1})=\{u_{6k+1},u_{6k+4}\}$.
If $\widetilde{Y}\xrightarrow{\text{Reg}}u_{3k}$, then $\widetilde{Y}=u_{6k+4}$, and thus $\widetilde{X}\xrightarrow{\text{Reg}}u_{6k-2}$, implying that $\widetilde{X}\in C(u_{5})\cap C(u_{6k-2})=\{u_{3k-2},u_{3k+1},u_{3k+4},u_{3k+7}\}$.
Base on symmetry, it suffices to discuss the following possibilities:
\begin{equation}\label{f1}
\widetilde{W}=\{u_1,u_{3k-2},u_{6k+4}\},\,\{u_1,u_{3k+1},u_{6k+4}\}\,\text{or}\,\{u_1,u_{3k-2},u_{6k+1}\}.
\end{equation}
We claim that $\widetilde{W}\neq\{u_1,u_{3k+1},u_{6k+4}\}$, since if otherwise, $W$ cannot resolve $u_{3k-9}$ and $u_{3k+13}$, therefore it is not a resolving set.

Now suppose that $\widetilde{W}=\{u_1,u_{3k-2},u_{6k+4}\}$. 
Because $\{O,v_{3k-2},Y\}$ cannot resolve $u_{3k-1}$ and $u_{3k+1}$, we get 
$X=u_{3k-2}$. For the remaining $4$ cases, one can verify that
\begin{align*}
	&r(v_2|\{u_1,u_{3k-2},u_{6k+4}\})=r(v_{6k+2}|\{u_1,u_{3k-2},u_{6k+4}\})=(2,k+1,3),\\
	&r(v_2|\{u_1,u_{3k-2},v_{6k+4}\})=r(v_{6k+2}|\{u_1,u_{3k-2},v_{6k+4}\})=(2,k+1,4),\\
	&r(v_{3k-4}|\{v_1,u_{3k-2},u_{6k+4}\})=r(v_{3k+4}|\{v_1,u_{3k-2},u_{6k+4}\})=(k+1,3,k+1),\\
	&r(u_{6k+1}|\{v_1,u_{3k-2},v_{6k+4}\})=r(u_{6k+3}|\{v_1,u_{3k-2},v_{6k+4}\})=(3,k+3,2).
\end{align*}
Thus there are always two vertices that cannot be resolved by $W$, a contradiction.

Suppose next that $\widetilde{W}=\{u_1,u_{3k-2},u_{6k+1}\}$. Because $\{O,v_{3k-2},Y\}$ cannot resolve $u_{3k}$ and $u_{3k+2}$, we easily deduce that $X=u_{3k-2}$. For the remaining $4$ cases, one can verify that
\begin{align*}
	&r(v_{6k+2}|\{u_1,u_{3k-2},u_{6k+1}\})=r(v_{6k+4}|\{u_1,u_{3k-2},u_{6k+1}\})=(2,k+1,2),\\
	&r(v_{3k-3}|\{u_1,u_{3k-2},v_{6k+1}\})=r(v_{3k+1}|\{u_1,u_{3k-2},v_{6k+1}\})=(k+1,2,k),\\
	&r(v_{3k-4}|\{v_1,u_{3k-2},u_{6k+1}\})=r(v_{3k+2}|\{v_1,u_{3k-2},u_{6k+1}\})=(k+1,3,k+2),\\
	&r(u_{6k+2}|\{v_1,u_{3k-2},v_{6k+1}\})=r(u_{6k+4}|\{v_1,u_{3k-2},v_{6k+1}\})=(2,k+2,2).
\end{align*}
Thus there are always two vertices that cannot be resolved by $W$, a contradiction.

$\mathbf{Case\, 2.}$ $d^{*}(\widetilde{O},\widetilde{X})\equiv d^{*}(\widetilde{O},\widetilde{Y})\equiv1\,(\text{mod}\,3)$. Assume that $d^{*}(\widetilde{O},\widetilde{X})<d^{*}(\widetilde{O},\widetilde{Y})$. One can verify that $d^{*}(\widetilde{X},\widetilde{Y})\equiv d^{*}(\widetilde{X},\widetilde{O})\equiv0\,(\text{mod}\,3)$, so that Case 2 can be reduced to Case 1.

$\mathbf{Case\, 3.}$ $d^{*}(\widetilde{O},\widetilde{X})\equiv d^{*}(\widetilde{O},\widetilde{Y})\equiv2\,(\text{mod}\,3)$. Without loss of generality, we assume that  $\widetilde{X}\xrightarrow{\text{Reg}}u_{6}$, namely $\widetilde{X}\in \{u_3,u_6,u_9\}$. Since $B(u_6)\cap B(u_{3k+2})=\emptyset$, we see that $\widetilde{Y}\xrightarrow{\text{Reg}}u_{3k+2}$, namely $\widetilde{Y}\in  \{u_{3k},u_{3k+6}\}$. 
It follows from the relation $\left(B(u_6)\cup B(u_{3k+2})\right)\cap B(u_{3k-1})=\emptyset$ that neither $\widetilde{X}$ nor $\widetilde{Y}$ can recognize $u_{3k-1}$, so $W$ cannot possibly be a resolving set.

$\mathbf{Case\, 4.}$ $d^{*}(\widetilde{O},\widetilde{X})\equiv1\,(\text{mod}\,3)$ and $ d^{*}(\widetilde{O},\widetilde{Y})\equiv2\,(\text{mod}\,3)$. 
Suppose first that $\widetilde{X}\xrightarrow{\text{Reg}}u_{5}$, namely $\widetilde{X}\in \{u_2,u_5,u_8\}$. We use $\widetilde{X}\notin A(u_{12})$ to see that $\widetilde{Y}\xrightarrow{\text{Reg}}u_{12}$, so that 
$\widetilde{Y}\in\{u_3,u_6,u_9,u_{12},u_{15}\}$. Consequently $\widetilde{X}\xrightarrow{\text{Reg}}u_{3k+3i+1}$ for each $1\leq i\leq k-1$, then it follows that $\widetilde{X}=u_2$. Now we observe that $\widetilde{Y}\xrightarrow{\text{Reg}}u_{3k+5}$, and therefore $\widetilde{Y}=u_3$.
This argument yields the following possibility:
\begin{equation}\label{f2}
	\widetilde{W}=\{u_1,u_2,u_3\}.
\end{equation}
At this point, $W$ cannot resolve $u_{3k+2}$ and $u_{3k+6}$, and is not a resolving set.

Suppose next that $\widetilde{Y}\xrightarrow{\text{Reg}}u_{5}$, namely $\widetilde{Y}\in \{u_3,u_9,u_{3k+9},u_{3k+12},\ldots,u_{6k+3}\}$.
Note that $u_8$ must be recognized by $\widetilde{X}$ or $\widetilde{Y}$.
If $\widetilde{X}\xrightarrow{\text{Reg}}u_{8}$, then the relation
$\widetilde{X}\notin\cup_{i=4}^{k-1}A(u_{3k+3i+2})$ implies that $\widetilde{Y}\xrightarrow{\text{Reg}}u_{3k+3i+2}$ for each $4\leq i\leq k-1$, therefore $\widetilde{Y}\in\{u_3,u_9\}$.
Now we see that $\widetilde{X}\xrightarrow{\text{Reg}}u_{3k+4}$, so that $\widetilde{X}=u_2$. Then $\widetilde{Y}\xrightarrow{\text{Reg}}u_{3k+5}$,
so that $\widetilde{Y}=u_3$.
At this point, $\widetilde{W}$ can only take vertices in the manner described in Possibility (\ref{f2}), and there is nothing left for us to prove.
On the other hand, if $\widetilde{Y}\xrightarrow{\text{Reg}}u_{8}$, then $\widetilde{Y}\in B(u_5)\cap B(u_8)=\{u_{3k+12},u_{3k+15},\ldots,u_{6k+3}\}$.
It follows that
$\widetilde{X}\xrightarrow{\text{Reg}}u_{3k+2}$ and that $\widetilde{X}\xrightarrow{\text{Reg}}u_{3k}$, implying $\widetilde{X}\in\{u_{3k-1},u_{3k+2},u_{3k+5}\}$.
This yields $\widetilde{Y}\xrightarrow{\text{Reg}}u_{6k-2}$, so that $\widetilde{Y}\in\{u_{6k-3},u_{6k},u_{6k+3}\}$.
If $\widetilde{Y}=u_{6k-3}$, then $\widetilde{X}\xrightarrow{\text{Reg}}u_{3k-4}$, so that $\widetilde{X}=u_{3k-1}$; if $\widetilde{Y}=u_{6k}$, then $\widetilde{X}\xrightarrow{\text{Reg}}u_{3k-1}$, so that $\widetilde{X}\in\{u_{3k-1},u_{3k+2}\}$.  
Due to symmetry, we conclude with the following possibilities:
\begin{equation}\label{f3}
	\widetilde{W}=\{u_1,u_{3k-1},u_{6k+3}\},\{u_1,u_{3k+2},u_{6k+3}\},\{u_1,u_{3k-1},u_{6k}\}\,\text{or}\,\{u_1,u_{3k-1},u_{6k-3}\}.
\end{equation}
The following table shows that regardless of whether $O$, $X$ and $Y$ are on the outer or inner circle, there exists a pair of vertices that $W$ cannot resolve, and thus $W$ cannot serve as a resolving set for the graph.
\begin{table}[h]
	\centering
	\begin{tabular}{|c|c|}
		\hline
		$\widetilde{W}$& \textbf{Vertex pairs indistinguishable by $W$}\\\hline
		\hline
		$\{u_1,u_{3k-1},u_{6k+3}\}$ & $u_{3k-9}$, $u_{3k+11}$ \\
		\hline
		$\{u_1,u_{3k+2},u_{6k+3}\}$ & $u_{3k}$, $u_{3k+4}$ \\
		\hline
		$\{u_1,u_{3k-1},u_{6k-3}\}$ &  $u_{3k-3}$, $u_{3k+1}$\\
		\hline
		$\{u_1,u_{3k-1},u_{6k}\}$ &  $u_{3k-3}$, $u_{3k+1}$\\
		\hline
	\end{tabular}
\end{table}

$\mathbf{Case\, 5.}$ $d^{*}(\widetilde{O},\widetilde{X})\equiv0\,(\text{mod}\,3)$ and $ d^{*}(\widetilde{O},\widetilde{Y})\equiv1\,(\text{mod}\,3)$. 
If $d^{*}(\widetilde{O},\widetilde{X})<d^{*}(\widetilde{O},\widetilde{Y})$, then $d^{*}(\widetilde{Y},\widetilde{O})\equiv d^{*}(\widetilde{Y},\widetilde{X})\equiv0\,(\text{mod}\,3)$, so that Case 5 can be reduced to Case 1. If $d^{*}(\widetilde{O},\widetilde{X})>d^{*}(\widetilde{O},\widetilde{Y})$, then $d^{*}(\widetilde{X},\widetilde{O})\equiv1 \,(\text{mod}\,3)$ and $d^{*}(\widetilde{X},\widetilde{Y})\equiv2 \,(\text{mod}\,3)$, so that Case 5 can be reduced to Case 4.

$\mathbf{Case\, 6.}$ $d^{*}(\widetilde{O},\widetilde{X})\equiv0\,(\text{mod}\,3)$ and $ d^{*}(\widetilde{O},\widetilde{Y})\equiv2\,(\text{mod}\,3)$. 
If $d^{*}(\widetilde{O},\widetilde{X})<d^{*}(\widetilde{O},\widetilde{Y})$, then $d^{*}(\widetilde{Y},\widetilde{O})\equiv d^{*}(\widetilde{Y},\widetilde{X})\equiv2\,(\text{mod}\,3)$, so that Case 6 can be reduced to Case 3. If $d^{*}(\widetilde{O},\widetilde{X})>d^{*}(\widetilde{O},\widetilde{Y})$, then $d^{*}(\widetilde{Y},\widetilde{X})\equiv1 \,(\text{mod}\,3)$ and $d^{*}(\widetilde{Y},\widetilde{O})\equiv2 \,(\text{mod}\,3)$, so that Case 6 can be reduced to Case 4.

We conclude that $W$ cannot be a resolving set of $P(n,3)$, and thus the metric dimension of $P(n,3)$ has a lower bound of $4$. The proof is complete.
\end{proof}

\subsection{Case when $n=6k+5$}
Throughout this subsection, we consistently assume that $n=6k+5$ and $k\geq6$. For $n\equiv 5\,(\text{mod}\,6)$, the distance between two vertices in $P(n,3)$ is
\begin{align*}
	&d(u_i,v_j)=
	\begin{cases}
		f(L)+1,    &\text{if}\,\, L\leq3k+1, \\
		k+2,    &\text{if}\,\, L=3k+2,
	\end{cases} \\
	&d(u_i,u_j)=
	\begin{cases}
		L,    &\text{if}\,\, L\leq2, \\
		f(L)+2,    &\text{if}\,\, 3\leq L\leq3k+1, \\
		k+3,    &\text{if}\,\, L=3k+2,
	\end{cases}\\
	&d(v_i,v_j)=
	\begin{cases}
		f(L),\quad   & \text{if } L\equiv 0 \,(\text{mod}\,3), \\
		k+2,\quad   & \text{if } L=3k-1, \\
		k+1,\quad   & \text{if } L=3k+2,\\
		f(L)+2,\quad   & \text{elsewise},
	\end{cases}
\end{align*}
where $L=|i-j|\wedge (n-|i-j|)$.
According to this distance formula, we can find out all good vertices for $u_1$:
$$
u_5,u_6,u_8,u_9,\ldots,u_{3i-1},u_{3i},\ldots,u_{6k-4},u_{6k-3},u_{6k-1},u_{6k},u_{6k+5}.
$$
Noticing that the set of good vertices for $v_1$ is deduced from that for $u_1$ by adding $4$ new vertices $u_2,u_3,u_{6k+2},u_{6k+3}$. Hence each vertex on the outer cycle can recognize more vertices than the corresponding one on the inner cycle. 
Using the approach detailed in Subsection \ref{subsection3}, for any given vertex, we can identify all the vertices within sets $A$, $B$, or $C$ that are capable of recognizing it (see Table \ref{table3} for details).  Moving forward, we present the following theorem.
\begin{table}[h]
	\centering
	\begin{tabular}{|c|c|c|}
		\hline 
		\textbf{Set} & \textbf{Vertices} & \textbf{Range}\\ \hline
		\hline
		$A(u_{3i-1})$ & $u_2,u_5,\ldots,u_{3i+2}$&$2\leq i\leq 2k$ \\
		\hline
		$A(u_{3i})$ & $u_{3i-1},u_{3i+2},\ldots,u_{6k+5}$&$2\leq i\leq 2k$ \\
		\hline
		$B(u_{3i-1})$ & $u_{3i-3},u_{3i+3}$ &$2\leq i\leq 2k$\\
		\hline
		$B(u_{3i})$ & $u_{3},u_{6},\ldots,u_{3i+3}$&$2\leq i\leq 2k$ \\
		\hline
		$C(u_{3i-1})$ & $u_{3i-2},u_{3i+1},\ldots,u_{6k+4}$&$2\leq i\leq 2k$ \\
		\hline
		$C(u_{3i})$ & $u_{3i-2},u_{3i+4}$&$2\leq i\leq 2k$ \\
		\hline
		\end{tabular}
	\caption{The result when $n=6k+5$.}\label{table3}
\end{table}

\begin{theorem}\label{5}
	If $n=6k+5$ and $k\geq6$, then $\dim(P(n,3))\geq4$.
\end{theorem}
\begin{proof}
Suppose on the contrary that $W:=\{O,X,Y\}$ is a resolving set of $P(n,3)$.	
A proof analogous to that in Theorem \ref{3} shows no two vertices of $W$ have the same subscript.	
Then let $\widetilde{O},\widetilde{X},\widetilde{Y}$ be as defined in Theorem \ref{3}. It is clear that $\widetilde{W}:=\{\widetilde{O},\widetilde{X},\widetilde{Y}\}$ can recognize all the vertices on the outer cycle.
Like before, we will discuss this problem in six cases.	Apart from Cases $5,6$, we assume $\widetilde{O}=u_1$
	
$\mathbf{Case\, 1.}$ $d^{*}(\widetilde{O},\widetilde{X})\equiv d^{*}(\widetilde{O},\widetilde{Y})\equiv0\,(\text{mod}\,3)$. 
It follows that $\{\widetilde{X},\widetilde{Y}\}$ can recognize $u_{3i}$ for $2\leq i\leq6$.
Suppose that $\widetilde{X}\xrightarrow{\text{Reg}}u_{6}$.
A proof analogous to that in Theorem \ref{3} shows $\widetilde{X}\in C(u_{6})\cap C(u_{18})$, which is a contradiction.

$\mathbf{Case\, 2.}$ $d^{*}(\widetilde{O},\widetilde{X})\equiv d^{*}(\widetilde{O},\widetilde{Y})\equiv2\,(\text{mod}\,3)$. 
We see that $\{\widetilde{X},\widetilde{Y}\}$ can recognize $u_{3i-1}$ for $2\leq i\leq6$.
If we assume $\widetilde{X}\xrightarrow{\text{Reg}}u_{5}$, then a similar proof can be used to deduce that $\widetilde{X}\in B(u_{5})\cap B(u_{17})$, which is a contradiction.

$\mathbf{Case\, 3.}$ $d^{*}(\widetilde{O},\widetilde{X})\equiv d^{*}(\widetilde{O},\widetilde{Y})\equiv1\,(\text{mod}\,3)$. 
Based on the symmetry, we assume that $\widetilde{X}\xrightarrow{\text{Reg}}u_{5}$, namely $\widetilde{X}\in A(u_5)=\{u_2,u_5,u_8\}$.
It follows from the relation $A(u_5)\cap A(u_{6k})=\emptyset$ that $\widetilde{Y}\xrightarrow{\text{Reg}}u_{6k}$, namely $\widetilde{Y}\in A(u_{6k})=\{u_{6k-1},u_{6k+2},u_{6k+5}\}$.

We shall demonstrate that $\widetilde{X}\neq u_8$. To do this, we use proof by contradiction, assuming instead that $\widetilde{X}= u_8$. Because $O$, $Y$, and $v_8$ cannot recognize $u_5$, it follows that $X=u_8$. Since 
$X(=u_8)$, $Y$, and $u_1$ cannot resolve $v_2$ and $v_4$, we have $O=v_1$.
At this point, if $Y$ is on the inner cycle, then $W$ cannot recognize $u_2$; if $Y$ is on the outer cycle, then $W$ cannot resolve $v_{6k+4}$ and $v_3$. Now that we have proven $\widetilde{X}\neq u_8$, by symmetry, we can similarly prove that $\widetilde{Y}\neq u_{6k-1}$. We can also rule out the possibility that $\widetilde{W}=\{u_1,u_2,u_{6k+5}\}$, because otherwise,
$W$ cannot resolve $u_{3k+1}$ and $u_{3k+6}$.

Now let us discuss the case where $\widetilde{W}=\{u_1,u_5,u_{6k+2}\}$. One can verify that $\{u_1,u_5,Y\}$ cannot resolve $v_2$ and $v_4$, and that $\{v_1,v_5,Y\}$ cannot resolve $u_2$ and $u_4$, implying that $O$ and $X$ must be on different cycles. By symmetry, it can be similarly proved that $O$ and $Y$ must be on different cycles. Therefore, we have deduced two cases: $W=\{u_1,v_5,v_{6k+2}\}$ or $\{v_1,u_5,u_{6k+2}\}$.
In the first case, $W$ cannot resolve $u_3$ and $u_{6k+4}$, while in the second case, it cannot resolve $v_3$ and $v_{6k+4}$, always leading to a contradiction.
  
Finally, based on symmetry, it suffices to consider the case $\widetilde{W}=\{u_1,u_5,u_{6k+5}\}$. 
Using a proof similar to that in the previous paragraph, it is possible to rule out the cases where $\{O,X\}=\{u_1,u_5\}$ and $W=\{v_1,v_5,v_{6k+5}\}$,
and the following table shows that in the remaining $5$ cases, $W$ cannot serve as a resolving set.
\begin{table}[h]
	\centering
	\begin{tabular}{|c|c||c|c|}
		\hline
		$\{O,X,Y\}$& \textbf{unresolveable pairs}&$\{O,X,Y\}$ & \textbf{unresolveable pairs}\\\hline
		\hline
		$\{u_1,v_5,u_{6k+5}\}$ & $v_{3k+3},v_{3k+4}$&$\{v_1,u_5,u_{6k+5}\}$ & $v_{3},v_{6k+4}$ \\
		\hline
		$\{u_1,v_5,v_{6k+5}\}$ & $u_{3},u_{6k+4}$&	$\{v_1,u_5,v_{6k+5}\}$ & $u_{3k+3},v_{3k+7}$ \\
		\hline
		$\{v_1,v_5,u_{6k+5}\}$ & $v_{3k},v_{3k+7}$ &\multicolumn{2}{c|}{ }\\
		\hline
		\end{tabular}
\end{table}

$\mathbf{Case\, 4.}$ $d^{*}(\widetilde{O},\widetilde{X})\equiv0\,(\text{mod}\,3)$ and $ d^{*}(\widetilde{O},\widetilde{Y})\equiv1\,(\text{mod}\,3)$. 
Since $C(u_{6k})\cap C(u_{6k-3})=\emptyset$, at least one of $u_{6k}$ and $u_{6k-3}$ can be recognized by $\widetilde{Y}$, so that $\widetilde{Y}\in\{u_{6k-4},u_{6k-1},u_{6k+2},u_{6k+5}\}$.

Suppose first that $\widetilde{Y}=u_{6k-4}$. It follows that $\widetilde{X}\xrightarrow{\text{Reg}}u_{6k}$, implying that $\widetilde{X}\in\{u_{6k-2},u_{6k+4}\}$.
Because $\{X,u_{6k-4}\}$ cannot distinguish between $u_{6k+3}$ and $u_{6k+5}$, nor between $v_{6k+3}$ and $v_{6k+5}$, while $O$ can only distinguish between one of the pairs, we thus have $Y=v_{6k-4}$.
It is easy to verify that $\{O,X,Y=v_{6k-4}\}$ cannot recognize $u_{6k-3}$, leading to the contradiction.
 
Suppose next that $\widetilde{Y}=u_{6k-1}$. It follows from $\widetilde{Y}\notin A(u_{6k-7})$ that $\widetilde{X}\xrightarrow{\text{Reg}}u_{6k-7}$, so $\widetilde{X}\in\{u_{6k-8},u_{6k-5},u_{6k-2},u_{6k+1},u_{6k+4}\}$. 
We first exclude the case where $\widetilde{X}=u_{6k+1}$. Suppose on the contrary that $\widetilde{X}=u_{6k+1}$.
Because $\{u_1,X\}$ cannot distinguish between $u_{6k}$ and $u_{6k+2}$, nor between $v_{6k}$ and $v_{6k+2}$, while $Y$ can only distinguish between one of the pairs, we thus have $O=v_1$, and therefore $Y=u_{6k-1}$. One can verify that $\{v_1,u_{6k+1},u_{6k-1}\}$ cannot resolve $v_{3k-3}$ and $v_{3k}$, and that $\{v_1,v_{6k+1},u_{6k-1}\}$ cannot resolve $u_{6k+3}$ and $u_{6k+5}$, yielding the contradiction.
Because $\{X,u_{6k-1}\}$ cannot distinguish between $u_{6k+3}$ and $u_{6k+5}$, nor between $v_{6k+3}$ and $v_{6k+5}$, while $O$ can only distinguish between one of the pairs, we thus have $Y=v_{6k-1}$, and therefore $O=u_{1}$.
At this point, neither $O=u_1$ nor $Y=v_{6k-1}$ can recognize $u_{6k}$, while $X$ can only recognize $u_{6k}$ when it is equal to $u_{6k-2}$ or $u_{6k+4}$. On the other hand, $\{u_1,u_{6k-2},v_{6k-1}\}$ cannot resolve $v_{6k-3}$ and $v_{6k+1}$, and $\{u_1,u_{6k+4},v_{6k-1}\}$ cannot resolve
$v_{3k}$ and $v_{3k+3}$. There are always two vertices that cannot be distinguished by $W$, hence $W$ is not a resolving set, leading to the contradiction.

Now suppose that $\widetilde{Y}=u_{6k+2}$. It follows from $\widetilde{Y}\notin A(u_{6k-4})$ that $\widetilde{X}\xrightarrow{\text{Reg}}u_{6k-4}$, namely $\widetilde{X}\in\{u_{6k-5},u_{6k-2},u_{6k+1},u_{6k+4}\}$.
If $\widetilde{X}=u_{6k+4}$, then $\{O,X,Y\}$ cannot resolve $u_{6k-3}$ and $u_6$, therefore the case where $\widetilde{X}=u_{6k+4}$ can be excluded.
If $\widetilde{X}\in\{u_{6k-5},u_{6k-2}\}$, then the pairs of vertices $u_{6k+3}$ and $u_{6k+5}$, and $v_{6k+3}$ and $v_{6k+5}$ cannot be distinguished by $X$, which implies that $O$ and $Y$ must be on different cycles. Rows $2$ to $9$ of Table \ref{tttt1} indicate that in the remaining $16$ cases, there exist two vertices that cannot be resolved by $W$, hence $W$ is not a resolving set, leading to the contradiction.

Finally, suppose that $\widetilde{Y}=u_{6k+5}$. It follows from $\widetilde{Y}\notin A(u_{6k-1})$ that $\widetilde{X}\xrightarrow{\text{Reg}}u_{6k-1}$, so $\widetilde{X}\in\{u_{6k-2},u_{6k+1},u_{6k+4}\}$.
Because $\{O,u_{6k+4},v_{6k+4},Y\}$ cannot resolve $u_{3k+2}$ and $u_{3k+3}$, we obtain $\widetilde{X}\neq u_{6k+4}$.
Since $\{v_1,u_{6k-2}\}$ cannot resolve $u_{6k+2}$ and $u_{6k+4}$, nor can it resolve $v_{6k+2}$ and $v_{6k+4}$, and since 
$Y$ can only distinguish one out of the two pairs of vertices, we see that if $X=u_{6k-2}$, then $O=u_1$, and therefore $Y=v_{6k+5}$. On the other hand, if $X=v_{6k-2}$, then $\{O,Y\}\neq\{v_1,v_{6k+5}\}$, this is because $\{v_1,v_{6k-2},v_{6k+5}\}$ cannot resolve $u_{6k+2}$ and $u_{6k+4}$.
Rows $10$ to $15$ of Table \ref{tttt1} demonstrate that in the remaining $12$ cases, $W$ cannot be the resolving set.

\begin{table}[h]
	\centering
	\begin{tabular}{|c|c||c|c|}
		\hline
		$\{O,X,Y\}$& \textbf{unresolveable pairs}&$\{O,X,Y\}$& \textbf{unresolveable pairs}\\\hline
		\hline
		$\{v_1,u_{6k-5},u_{6k+2}\}$ & $v_{6k-4},v_{6k-2}$&$\{v_1,u_{6k+1},v_{6k+2}\}$ & $v_{3k-4},v_{3k+3}$ \\
		\hline
		$\{v_1,v_{6k-5},u_{6k+2}\}$ & $u_{6k-4},u_{6k-2}$&$\{v_1,u_{6k+1},u_{6k+2}\}$ & $v_{3k-4},v_{3k+3}$ \\
		\hline
		$\{u_1,u_{6k-5},v_{6k+2}\}$ & $v_{6k},v_{6k+4}$&$\{v_1,v_{6k+1},u_{6k+2}\}$ & $u_{6k},u_{6k+4}$ \\
		\hline
		$\{u_1,v_{6k-5},v_{6k+2}\}$ & $u_{6k-4},u_{6k-2}$&$\{v_1,v_{6k+1},v_{6k+2}\}$ & $u_{6k+3},u_{6k+5}$ \\
		\hline
		$\{v_1,u_{6k-2},u_{6k+2}\}$ & $v_{6k-1},v_{6k+1}$&$\{u_1,u_{6k+1},v_{6k+2}\}$ & $v_{6k},v_{6k+4}$ \\
		\hline
		$\{v_1,v_{6k-2},u_{6k+2}\}$ & $u_{6k},u_{6k+4}$&$\{u_1,u_{6k+1},u_{6k+2}\}$ & $v_{6k+3},v_{6k+5}$ \\
		\hline
		$\{u_1,u_{6k-2},v_{6k+2}\}$ & $v_{6k},v_{6k+4}$&$\{u_1,v_{6k+1},u_{6k+2}\}$ & $v_{6k+3},v_{6k+5}$ \\
		\hline
		$\{u_1,v_{6k-2},v_{6k+2}\}$ & $u_{6k-1},u_{6k+1}$&$\{u_1,v_{6k+1},v_{6k+2}\}$ & $u_{3k-1},v_{3k+3}$ \\\hline
		\hline
		$\{u_1,u_{6k-2},v_{6k+5}\}$ & $v_{6k-10},u_{6k-7}$&$\{v_1,u_{6k+1},v_{6k+5}\}$ & $u_{3k+3},v_{3k-1}$  \\
		\hline
		$\{u_1,v_{6k-2},v_{6k+5}\}$ & $v_{6k},v_{5}$&$\{u_1,v_{6k+1},u_{6k+5}\}$ & $v_{6k},v_{5}$  \\
		\hline
		$\{u_1,v_{6k-2},u_{6k+5}\}$ & $u_{6k-1},u_{6k+1}$& 	$\{u_1,v_{6k+1},v_{6k+5}\}$ & $v_{6k},v_{5}$  \\
		\hline
		$\{v_1,v_{6k-2},u_{6k+5}\}$ & $u_{6k+3},u_{2}$& 	$\{v_1,v_{6k+1},u_{6k+5}\}$ & $u_{6k+3},u_{2}$ \\
		\hline
		$\{u_1,u_{6k+1},u_{6k+5}\}$ & $v_{6k+3},v_{2}$&  	$\{v_1,v_{6k+1},v_{6k+5}\}$ & $u_{6k+3},u_{2}$  \\
		\hline
		$\{u_1,u_{6k+1},v_{6k+5}\}$ & $v_{6k+3},v_{2}$&$\{v_1,u_{6k+1},u_{6k+5}\}$ & $v_{6k+2},v_{6k+4}$  \\
		\hline
	\end{tabular}
\caption{From lines $2$ to $9$, $\widetilde{Y}=u_{6k+2}$; from lines $10$ to $15$, $\widetilde{Y}=u_{6k+5}$}\label{tttt1}
\end{table}

$\mathbf{Case\, 5.}$ $d^{*}(\widetilde{O},\widetilde{X})\equiv0\,(\text{mod}\,3)$ and $ d^{*}(\widetilde{O},\widetilde{Y})\equiv2\,(\text{mod}\,3)$. 
If $d^{*}(\widetilde{O},\widetilde{X})<d^{*}(\widetilde{O},\widetilde{Y})$,
then $d^{*}(\widetilde{Y},\widetilde{O})\equiv d^{*}(\widetilde{Y},\widetilde{X})\equiv0\,(\text{mod}\,3)$, so that Case 5 can be reduced to Case 1. If $d^{*}(\widetilde{O},\widetilde{X})>d^{*}(\widetilde{O},\widetilde{Y})$,
then $d^{*}(\widetilde{Y},\widetilde{X})\equiv1\,(\text{mod}\,3)$ and $ d^{*}(\widetilde{Y},\widetilde{O})\equiv0\,(\text{mod}\,3)$, therefore Case 5 can be reduced to Case 4.

$\mathbf{Case\, 6.}$ $d^{*}(\widetilde{O},\widetilde{X})\equiv1\,(\text{mod}\,3)$ and $ d^{*}(\widetilde{O},\widetilde{Y})\equiv2\,(\text{mod}\,3)$.
If $d^{*}(\widetilde{O},\widetilde{X})<d^{*}(\widetilde{O},\widetilde{Y})$, then $d^{*}(\widetilde{X},\widetilde{Y})\equiv d^{*}(\widetilde{X},\widetilde{O})\equiv1\,(\text{mod}\,3)$, so that Case 6 can be reduced to Case 3. If $d^{*}(\widetilde{O},\widetilde{X})>d^{*}(\widetilde{O},\widetilde{Y})$, then $d^{*}(\widetilde{X},\widetilde{O})\equiv1\,(\text{mod}\,3)$ and $ d^{*}(\widetilde{X},\widetilde{Y})\equiv0\,(\text{mod}\,3)$, so that Case 6 can be reduced to Case 4.

We conclude that $W$ cannot be a resolving set of $P(n,3)$, and thus the metric dimension of $P(n,3)$ has a lower bound of $4$, which completes our proof.
\end{proof}

\subsection{Case when $n=6k+2$}
In this subsection and the following section, we always assume $n=6k+2$ and $k\geq6$. 
For $n\equiv 2\,(\text{mod}\,6)$, the distance formulas for $u_i$ and $u_j$, as well as for $u_i$ and $v_j$, are the same as those when $n\equiv 3\,(\text{mod}\,6)$, but the distance formula for $v_i$ and $v_j$ is slightly different:
\begin{align*}
	&d(v_i,v_j)=
	\begin{cases}
		f(L),\quad   & \text{if}\,\, L\equiv 0\,(\text{mod}\,3), \\
		k+1,\quad   & \text{if}\,\, L=3k-1,\\
		f(L)+2,\quad   & \text{elsewise}, \\
	\end{cases}
\end{align*}
where $L=|i-j|\wedge (n-|i-j|)$.
According to the distance formula, we can find out all good vertices for $u_1$:
\begin{equation*}
	\resizebox{.95\hsize}{!}{$u_5,u_6,\ldots,u_{3i-1},u_{3i},\ldots,u_{3k-1},u_{3k},u_{3k+1},u_{3k+2},u_{3k+3},\ldots,u_{3k+3j+2},u_{3k+3j+3},\ldots,u_{6k-4},u_{6k-3},u_{6k+2}.
		$}
\end{equation*}
For the vertex good for $u_1$, we can identify all the vertices within sets $A$, $B$, or $C$ that are capable of recognizing it (see Table \ref{tableX} for details).
\begin{table}[h]
	\centering
	\begin{tabular}{|c|c|c|}
		\hline
		\textbf{Set} & \textbf{Vertices} & \textbf{Range}\\\hline
		\hline
		$A(u_{3i-1})$ & $u_2,u_5,\ldots,u_{3i+2}$&$2\leq i\leq k$ \\
		\hline
		$A(u_{3i})$ & $u_{3i-1},u_{3i+2},\ldots,u_{3k+3i-1},u_{3k+3i+5},\ldots,u_{6k+2}$&$2\leq i\leq k-1$ \\
		\hline
		$A(u_{3k})$ & $u_{3k-1},u_{3k+2},\ldots,u_{6k-1}$& \\
		\hline
		$A(u_{3k+1})$ & $u_{3k-1},u_{3k+5}$& \\
		\hline
		$A(u_{3k+2})$ & $u_{5},u_{8},\ldots,u_{3k+5}$& \\
		\hline
		$A(u_{3k+3})$ & $u_{3k+2},u_{3k+5},\ldots,u_{6k+2}$& \\
		\hline
		$A(u_{3k+3i+2})$ & $u_2,\ldots,u_{3i-1},u_{3i+5},\ldots,u_{3k+3i+5}$ &$1\leq i\leq k-2$\\
		\hline
		$A(u_{3k+3i+3})$ & $u_{3k+3i+2},u_{3k+3i+5},\ldots,u_{6k+2}$&$1\leq i\leq k-2 $\\
		\hline
		$B(u_{3i-1})$ & $u_{3i-3},u_{3i+3}$&$2\leq i\leq k+1$ \\
		\hline
		$B(u_{3i})$ & $u_{3},u_{6},\ldots,u_{3i+3}$&$2\leq i\leq k$ \\
		\hline
		$B(u_{3k+1})$ & $u_{3k},u_{3k+3},\ldots,u_{6k}$& \\
		\hline
		$B(u_{3k+3})$ & $u_{6},u_{9},\ldots,u_{3k+6}$& \\
		\hline
		$B(u_{3k+3i+2})$ & $u_{3k+3i},u_{3k+3i+6}$ &$1\leq i\leq k-2$\\
		\hline
		$B(u_{3k+3i+3})$ & $u_{3},u_{6},\ldots,u_{3i},u_{3i+6},\ldots,u_{3k+3i+6}$&$1\leq i\leq k-2 $\\
		\hline
	    $C(u_{3i-1})$ & $u_{3i-2},\ldots,u_{3k+3i-2},u_{3k+3i+4},\ldots,u_{6k+1}$&$2\leq i\leq k-1$ \\
     	\hline
	    $C(u_{3i})$ & $u_{3i-2},u_{3i+4}$&$2\leq i\leq k+1$ \\
	    \hline
	    $C(u_{3k-1})$ & $u_{3k-2},u_{3k+1},\ldots,u_{6k-2}$& \\
     	\hline
     	$C(u_{3k+1})$ & $u_{4},u_{7},\ldots,u_{3k+4}$& \\
    	\hline
    	$C(u_{3k+2})$ & $u_{3k+1},u_{3k+4},\ldots,u_{6k+1}$& \\
	    \hline
     	$C(u_{3k+3i+2})$ & $u_{3k+3i+1},u_{3k+3i+4},\ldots,u_{6k+1}$ &$1\leq i\leq k-2$\\
     	\hline
	    $C(u_{3k+3i+3})$ & $u_{3k+3i+1},u_{3k+3i+7}$&$1\leq i\leq k-2 $\\
	   \hline
	\end{tabular}
	\caption{The result when $n=6k+2$.}\label{tableX}
\end{table}

\begin{theorem}\label{22222}
	If $n=6k+2$ and $k\geq6$, then $\dim(P(n,3))\geq4$.
\end{theorem}
\begin{proof}
	Suppose on the contrary that $W:=\{O,X,Y\}$ is a resolving set of $P(n,3)$.
	We claim that no two vertices of $W$ have the same subscript; since if otherwise, those two vertices with the same subscript can only recognize $2k+4$ vertices on the outer cycle, and the other vertex must recognize the remaining $4k-2$ vertices on the outer cycle, which is impossible.
	Let $\widetilde{O},\widetilde{X},\widetilde{Y}$ be defined in the same way as $\widetilde{O},\widetilde{X},\widetilde{Y}$ were in Theorem \ref{3}. We see that $\widetilde{W}:=\{\widetilde{O},\widetilde{X},\widetilde{Y}\}$ can still recognize all the vertices on the outer cycle.
	Let us discuss this problem in several cases. Apart from Cases $5,6$, we assume $\widetilde{O}=u_1$ by the rotational symmetry of $P(n,3)$.
	
	$\mathbf{Case\, 1.}$ $d^{*}(\widetilde{O},\widetilde{X})\equiv d^{*}(\widetilde{O},\widetilde{Y})\equiv0\,(\text{mod}\,3)$. It follows that $\{\widetilde{X},\widetilde{Y}\}$ can recognize $u_{3i}$ for $2\leq i\leq6$.
	Suppose that $\widetilde{X}\xrightarrow{\text{Reg}}u_{6}$.
	A proof analogous to that in Theorem \ref{3} shows $\widetilde{X}\in C(u_{6})\cap C(u_{18})$, which is a contradiction.
	
	$\mathbf{Case\, 2.}$ $d^{*}(\widetilde{O},\widetilde{X})\equiv d^{*}(\widetilde{O},\widetilde{Y})\equiv1\,(\text{mod}\,3)$. 
	Suppose that $\widetilde{X}\xrightarrow{\text{Reg}}u_{5}$, namely $\widetilde{X}\in \{u_2,u_5,u_8\}$. The relation $\widetilde{X}\notin A(u_{3k+1})$ implies that $\widetilde{Y}\xrightarrow{\text{Reg}}u_{3k+1}$, and therefore $\widetilde{Y}\in A(u_{3k+1})=\{u_{3k-1},u_{3k+5}\}$.
	We observe that neither vertex $\widetilde{X}$ nor $\widetilde{Y}$ belongs to set $A(u_{6k-3})$; consequently, $u_{6k-3}$ cannot be recognized by $\widetilde{W}$, a contradiction.
	
	$\mathbf{Case\, 3.}$ $d^{*}(\widetilde{O},\widetilde{X})\equiv d^{*}(\widetilde{O},\widetilde{Y})\equiv2\,(\text{mod}\,3)$. It follows that $\{\widetilde{X},\widetilde{Y}\}$ can recognize $u_{3i-1}$ for $2\leq i\leq6$.
	Suppose that $\widetilde{X}\xrightarrow{\text{Reg}}u_{5}$.
	A proof analogous to that in Theorem \ref{3} shows $\widetilde{X}\in B(u_{5})\cap B(u_{17})$, which is a contradiction.
	
	$\mathbf{Case\, 4.}$ $d^{*}(\widetilde{O},\widetilde{X})\equiv1\,(\text{mod}\,3)$ and $ d^{*}(\widetilde{O},\widetilde{Y})\equiv2\,(\text{mod}\,3)$. 
	Suppose first that $\widetilde{X}\xrightarrow{\text{Reg}}u_{5}$, namely $\widetilde{X}\in \{u_2,u_5,u_8\}$. Since $\widetilde{X}\notin A(u_{3k+1})\cup A(u_{3k-6})$, it follows that $\widetilde{Y}$ must recognize $u_{3k+1}$ and $u_{3k-6}$. However, $B(u_{3k+1})$ and $B(u_{3k-6})$ have no vertices in common, yielding a contradiction. 
	
	Suppose next that $\widetilde{Y}\xrightarrow{\text{Reg}}u_{5}$, namely $\widetilde{Y}\in \{u_3,u_9\}$. It follows from the relation $\widetilde{Y}\notin B(u_{8})\cup B(u_{3k+1})$ that $\widetilde{X}$ must recognize $u_{8}$ and $u_{3k+1}$. Likewise, $A(u_{8})$ and $A(u_{3k+1})$ share no common elements, leading to a contradiction.
	
	$\mathbf{Case\, 5.}$ $d^{*}(\widetilde{O},\widetilde{X})\equiv0\,(\text{mod}\,3)$ and $ d^{*}(\widetilde{O},\widetilde{Y})\equiv1\,(\text{mod}\,3)$. 
	If $d^{*}(\widetilde{O},\widetilde{X})<d^{*}(\widetilde{O},\widetilde{Y})$, then $d^{*}(\widetilde{X},\widetilde{Y})\equiv1\,(\text{mod}\,3)$ and 
	$d^{*}(\widetilde{X},\widetilde{O})\equiv2\,(\text{mod}\,3)$,
	so that Case 5 can be reduced to Case 4. 
	If $d^{*}(\widetilde{O},\widetilde{X})>d^{*}(\widetilde{O},\widetilde{Y})$, then $d^{*}(\widetilde{Y},\widetilde{O})\equiv1 \,(\text{mod}\,3)$ and $d^{*}(\widetilde{Y},\widetilde{X})\equiv2 \,(\text{mod}\,3)$, so that Case 5 can also be reduced to Case 4.
	
	$\mathbf{Case\, 6.}$ $d^{*}(\widetilde{O},\widetilde{X})\equiv0\,(\text{mod}\,3)$ and $ d^{*}(\widetilde{O},\widetilde{Y})\equiv2\,(\text{mod}\,3)$. 
	If $d^{*}(\widetilde{O},\widetilde{X})<d^{*}(\widetilde{O},\widetilde{Y})$,
	then $d^{*}(\widetilde{Y},\widetilde{O})\equiv d^{*}(\widetilde{Y},\widetilde{X})\equiv0\,(\text{mod}\,3)$, so that Case 6 can be reduced to Case 1. 
	If $d^{*}(\widetilde{O},\widetilde{X})>d^{*}(\widetilde{O},\widetilde{Y})$, then $d^{*}(\widetilde{X},\widetilde{Y})\equiv1 \,(\text{mod}\,3)$ and $d^{*}(\widetilde{X},\widetilde{O})\equiv2 \,(\text{mod}\,3)$, so that Case 6 can be reduced to Case 4.
	
	We conclude that $W$ cannot be a resolving set of $P(n,3)$, and thus the metric dimension of $P(n,3)$ has a lower bound of $4$. The proof is complete.
\end{proof}

Roughly speaking, when $n\equiv2 \,\,(\text{mod}\,\, 6)$, each $u_i$ recognizes fewer vertices than in other circumstances. Therefore, in order to recognize all vertices on the outer cycle, a resolving set requires ``more vertices'', which explains why proving the lower bound of the metric dimension to be $4$ when $n\equiv 2\,(\text{mod}\,6)$ is simpler than in other cases. 

\section{Upper bounds for the metric dimension}
In this section, we shall prove that, when $n\equiv2 \,\,(\text{mod}\,\, 6)$, 4 is also an upper bound for $\dim(P(n,3))$. It suffices to show that $W:=\{u_1,u_{3k-2},v_{3k-1},v_{6k+2}\}$ is a resolving set of $P(n,3)$.    
Using the distance formula, we can obtain the metric representations of all vertices with respect to $W$, as shown below.
\begin{align*}
	&r(u_{3i}|W)=
	\begin{cases}
		(2,k+1,k+1,2),\quad   & \text{if}\,\, i=1, \\
		(i+3,k-i+2,k-i+2,i+1),\quad   & \text{if}\,\, 2\leq i\leq k-2,\\
		(k+2,1,3,k),\quad   & \text{if}\,\, i=k-1,\\
		(k+3,2,2,k+1),\quad   & \text{if}\,\, i=k,\\
		(2k-i+3,i-k+4,i-k+2,2k-i+3),\quad   & \text{if}\,\, k+1\leq i\leq 2k-1, \\
		(3,k+2,k+2,3),\quad   & \text{if}\,\, i=2k. \\
	\end{cases}
\end{align*}
\begin{align*}
	&r(u_{3i+1}|W)=
	\begin{cases}
		(0,k+1,k+1,2),\quad   & \text{if}\,\, i=0, \\
		(i+2,k-i+1,k-i+1,i+2),\quad   & \text{if}\,\, 1\leq i\leq k-2,\\
		(k+1,0,2,k+1),\quad   & \text{if}\,\, i=k-1,\\
		(k+2,3,3,k+2),\quad   & \text{if}\,\, i=k,\\
		(2k-i+4,i-k+3,i-k+3,2k-i+2),\quad   & \text{if}\,\, k+1\leq i\leq 2k-1, \\
		(2,k+3,k+1,2),\quad   & \text{if}\,\, i=2k. \\
	\end{cases}
\end{align*}
\begin{align*}
	&r(u_{3i+2}|W)=
	\begin{cases}
		(1,k+2,k,3),\quad   & \text{if}\,\, i=0, \\
		(i+3,k-i+2,k-i,i+3),\quad   & \text{if}\,\, 1\leq i\leq k-3,\\
		(k+1,2,2,k+1),\quad   & \text{if}\,\, i=k-2,\\
		(k+2,1,1,k+2),\quad   & \text{if}\,\, i=k-1,\\
		(k+3,4,2,k+1),\quad   & \text{if}\,\, i=k,\\
		(2k-i+3,i-k+4,i-k+2,2k-i+1),\quad   & \text{if}\,\, k+1\leq i\leq 2k-1, \\
		(1,k+2,k+2,1),\quad   & \text{if}\,\, i=2k. \\
	\end{cases}
\end{align*}
\begin{align*}
	&r(v_{3i}|W)=
	\begin{cases}
		(i+2,k-i+1,k-i+3,i),\quad   & \text{if}\,\, 1\leq i\leq k-1,\\
		(k+2,3,3,k),\quad   & \text{if}\,\, i=k,\\
		(k+1,4,4,k+1),\quad   & \text{if}\,\, i=k+1,\\
		(2k-i+2,i-k+3,i-k+3,2k-i+4),\quad   & \text{if}\,\, k+2\leq i\leq 2k-1, \\
		(2,k+1,k+3,4),\quad   & \text{if}\,\, i=2k. \\
	\end{cases}
\end{align*}
\begin{align*}
	&r(v_{3i+1}|W)=
	\begin{cases}
		(i+1,k-i,k-i+2,i+3),\quad   & \text{if}\,\, 0\leq i\leq k-1,\\
		(k+1,2,4,k+3),\quad   & \text{if}\,\, i=k,\\
		(2k-i+3,i-k+2,i-k+4,2k-i+3),\quad   & \text{if}\,\, k+1\leq i\leq 2k-2, \\
		(4,k+1,k+1,4),\quad   & \text{if}\,\, i=2k-1,\\
		(3,k+2,k,3),\quad   & \text{if}\,\, i=2k. \\
	\end{cases}
\end{align*}
\begin{align*}
	&r(v_{3i+2}|W)=
	\begin{cases}
		(i+2,k-i+1,k-i-1,i+4),\quad   & \text{if}\,\, 0\leq i\leq k-2,\\
		(k+1,2,0,k+1),\quad   & \text{if}\,\, i=k-1,\\
		(k+2,3,1,k),\quad   & \text{if}\,\, i=k,\\
		(2k-i+2,i-k+3,i-k+1,2k-i),\quad   & \text{if}\,\, k+1\leq i\leq 2k-1, \\
		(2,k+1,k+1,0),\quad   & \text{if}\,\, i=2k. \\
	\end{cases}
\end{align*}

Note that in the metric representations above, the first and fourth coordinates, as well as the second and third coordinates, always differ by $2$. For vertices sharing the same metric representation, the magnitude relationship between the first and fourth coordinates, and between the second and third coordinates, is necessarily identical.
Vertices in $P(n,3)$ can therefore be categorized based on these relative magnitudes. By comparing metric representations within each category, the distinctness of all these metric representations can be easily verified, and thus 
$\{u_1,u_{3k-2},v_{3k-1},v_{6k+2}\}$ is a resolving set of $P(n,3)$. 

\section{Conclusion}
We have proved that, when $n\equiv2,3,4,5 \,\,(\text{mod}\,\, 6)$ and is sufficiently large, the metric dimension of generalized Petersen graphs $P(n,3)$ has a lower bound of $4$. Using the conclusions from \cite{Imran1} and this work, we can obtain the exact value of $\dim(P(n,3))$ when $n\equiv2,3,4,5\,\,(\text{mod}\,\, 6)$.

\noindent\textbf{Acknowledgements.} This research is supported by the National Natural Science Foundation of China (Nos. 12301606, 72301156).

\noindent\textbf{Statements and Declarations:} We declare that we do not have any commercial or associative interest that represents a conflict of 
interest in connection with the work submitted.

\begin{quote}
{\bf References and Notes}

\end{quote}
\end{document}